\documentclass[11pt]{amsart}
\usepackage[T1]{fontenc}
\usepackage[margin=1in]{geometry}
\usepackage{amscd,setspace,amssymb,amsopn,mathrsfs,lmodern,graphics,amsfonts,enumerate,verbatim,calc,times} 
\usepackage[dvipsnames]{xcolor}
\definecolor{RoyalBlue}{RGB}{0, 35, 102}

\usepackage[backend=biber, style=alphabetic, sorting=anyt]{biblatex}
\usepackage[bookmarks,colorlinks=true, citecolor=RoyalBlue, filecolor=RoyalBlue,linkcolor=RoyalBlue,urlcolor = RoyalBlue]{hyperref}
\usepackage{xurl}
\hypersetup{breaklinks=true}

\usepackage{amsmath}
\usepackage{amssymb}
\usepackage{amsthm}
\usepackage{stmaryrd}
\usepackage{mathtools}
\usepackage{mathrsfs}
\usepackage{esdiff}
\usepackage{graphicx}
\usepackage{comment}
\usepackage{float}

\usepackage{cleveref}

\usepackage{tikz}
\usepackage{tikz-cd}
\usepackage{pgfplots}
\pgfplotsset{compat=1.17}
\usetikzlibrary{positioning}
\usetikzlibrary{external, automata, arrows}
\tikzexternaldisable

\tikzset{
    labl/.style={anchor=south, rotate=270, inner sep=.5mm}
}

\definecolor{color1}{RGB}{250, 150, 150}
\definecolor{color2}{RGB}{150, 150, 250}

\theoremstyle{definition}
\newtheorem{theorem}{Theorem}[section]
\newtheorem*{theorem*}{Theorem}

\newtheorem*{maintheorem*}{Main Theorem}
\newtheorem{definition}[theorem]{Definition}

\newtheorem{lemma}[theorem]{Lemma}
\newtheorem{corollary}[theorem]{Corollary}

\newtheorem{claim}[theorem]{Claim}

\newenvironment{subproof}[1]
{
  
  \begin{proof}#1
}
{
\end{proof}

}  

\usepackage{ upgreek }

\newcommand*{\QQ}{\mathbb{Q}}
\newcommand*{\ZZ}{\mathbb{Z}}
\newcommand*{\NN}{\mathbb{N}}
\newcommand*{\FF}{\mathbb{F}}

\newcommand*{\hO}{\mathcal{O}}

\usepackage{enumitem}
\setlist[enumerate]{label=(\roman*),ref=(\roman*)}

\newcommand*{\st}{\ \vert\ }
\newcommand*{\blank}{{-}}
\newcommand*{\nth}{^{\textrm{th}}}

\DeclareMathOperator{\characteristic}{char}

\DeclareMathOperator{\Span}{Span}

\makeatletter
\title[Tight Closure in Line--\texorpdfstring{$S_4$}{} Quartics]{On Localization of Tight Closure in Line--\texorpdfstring{$S_4$}{} Quartics}

\author[Borevitz]{Levi Borevitz}
\address{Department of Mathematics, Williams College, Williamstown, MA 01267}
\email{\href{mailto:lyb1@williams.edu}{lyb1@williams.edu}}

\author[Nader]{Naima Nader}
\address{Mathematics and Statistics Department, Vassar College, Poughkeepsie, NY 12604}
\email{\href{mailto:nnader@vassar.edu}{nnader@vassar.edu}}

\author[Sandstrom]{Theodore J. Sandstrom}
\address{Department of Mathematics, Yale University, PO Box 208283, New Haven, CT 06520}
\email{\href{mailto:theo.sandstrom@yale.edu}{theo.sandstrom@yale.edu}}

\author[Shapiro]{Amelia Shapiro}
\address{Department of Mathematics, Brown University, Providence, RI 02912}
\email{\href{mailto:amelia\textunderscore shapiro@brown.edu}{amelia\textunderscore shapiro@brown.edu}}

\author[Simpson]{Austyn Simpson}
\address{Department of Mathematics, University of Michigan, Ann Arbor, MI 48109}
\thanks{Simpson was supported by NSF grant DMS \#2202890.}
\email{\href{mailto:austyn@umich.edu}{austyn@umich.edu}}

\author[Zomback]{Jenna Zomback}
\address{Department of Mathematics, Williams College, Williamstown, MA 01267}
\email{\href{mailto:jz8@williams.edu}{jz8@williams.edu}}

\begin{document}

\begin{abstract}
Building on work of Brenner and Monsky from 2010 and on a Hilbert--Kunz calculation of Monsky from 1998, we exhibit a novel example of a hypersurface over $\overline{\FF_2}$ in which tight closure does not commute with localization. Our methods involve a surprising tiling argument using Sierpi\'{n}ski triangles, as well as an inspection of a certain dynamical system in characteristic two.

\end{abstract}

\maketitle

\section{Introduction}
\label{sec:introduction}

Let $R$ be a noetherian domain of prime characteristic $p > 0$ and let $I \subseteq R$ be an ideal. An element $r \in R$ is in the \emph{tight closure of $I$}, denoted $I^*$, if there exists $0 \neq c \in R$ such that $c r^{p^e} \in I^{[p^e]} := \langle x^{p^e} \mid x \in I \rangle$ for all $e \geq 1$. This operation was introduced by Hochster and Huneke \cite{HH88,HH90} as a refinement of integral closure in prime characteristic and has since garnered immense interest among commutative algebraists and algebraic geometers. The theory initially emerged as a tool to unify and simplify existing Frobenius techniques, but in the subsequent decades it has been used to build a robust bridge via reduction mod $p$ to singularities in complex birational geometry (see e.g. \cite{HW02,Sch09,Smi97}).

Due to both the geometric impact of this closure theory and also the push to streamline ideas at the core of the subject, it was desirable for the tight closure operation $(\blank)^*$ to commute with localizing at a multiplicative subset $S\subseteq R$. This vexing question remained open for nearly two decades, and it was an active program to search for partial positive results (see \cite[Chapter 12]{Hun96} for a survey). However, Brenner and Monsky discovered the following surprising counterexample.
\begin{theorem}\cite[Theorem 2.5]{BM10}
Let 
\[ A = \frac{F[x,y,z,t]}{(t x^2 y^2 + z^4 + x y z^2 + x^3 z + y^3 z)}, \]
where $F = \overline{\FF_2}$ and let $I = (x^4, y^4, z^4)$. If $S = F[t] \setminus \{ 0 \}$, then $y^3 z^3 \in (S^{-1} I)^* \setminus S^{-1}(I^*)$.
\end{theorem}

The Brenner--Monsky hypersurface $A$ appearing above has since been used to motivate the study of Frobenius invariants in families \cite{CRST21,DS22,Smi20,Tri05}, though the underlying reasons for this noncommutativity remain mysterious. The absence of examples generated since \cite{BM10} indeed leaves open many questions as to how severe this obstruction might be (see \Cref{sec:questions} for more details). To this end the primary content of the present paper is the second example to the authors' knowledge demonstrating this pathology:

\begin{maintheorem*}\label{thm:mainthm}(\Cref{thm-main-thm})
Let
\[ R = \frac{F[x,y,z,t]}{(t z^4 + (x^2 + y z) (y^2 + x z))}, \]
where $F = \overline{\FF_2}$ and let $I = (x^4, y^4, z^4)$. If $S = F[t] \setminus \{ 0 \}$, then $y^3 z^3 \in (S^{-1}I)^* \setminus S^{-1}(I^*)$.
\end{maintheorem*}

The two hypersurfaces $A$ and $R$ above resemble each other closely (for example, both are normal hypersurfaces of dimension three whose defining polynomials agree when evaluated at $t = 1$), and the non-localization data therein is identical. The recent history of these rings is rooted in Hilbert--Kunz theory, a Frobenius-based multiplicity theory pioneered by Monsky in \cite{Mon83}. The Brenner--Monsky hypersurface and the hypersurface appearing in our Main Theorem originated in earlier papers of Monsky (dubbed ``point--$S_4$'' and ``line--$S_4$'' quartics in \cite{Mon98b} and \cite{Mon98a} respectively) where it was shown that the Hilbert--Kunz multiplicity may vary for the rings in question viewed as families in $t$. 

\cite{Mon98b} in particular (and \cite{BM10} by extension) depicts quantitatively a jump in the severity of singularity between the generic fiber and a special fiber, and there is recent evidence that such behavior is intricately linked to the tight closure localization problem (see e.g. \cite[Section 6]{Smi19} and \cite[Question 5.3]{Smi20}). By contrast, \cite{Mon98a} (and by extension our example $R$ which we contribute in this note) has the curiosity of being ``less geometric'' -- the enlargement of the tight closure after localizing is governed by a sparse (but still infinite) set of elements of $F$ which exhibit striking arithmetic dynamical properties that we will elucidate in later sections. In other words, the interesting phenomena for $(\blank)^*$ occur only at \emph{some} of the special fibers.

One final feature distinguishing $R$ from $A$ is that it appears to be \emph{slightly} less singular. If $g, h \in F[x,y,z,t]$ are the defining polynomials for $A$ and $R$ respectively, then $g + t x^2 y^2$ is a product of four linear forms whereas $h + t z^4$ is visibly a product of two irreducible quadratics. This sentiment is further bolstered by the Hilbert--Kunz multiplicity seemingly measuring slightly better singularities for specializations of $R$ versus specializations of $A$. The reader might then view our Main Theorem as a modest first step in finding rings with the mildest possible singularities but for which tight closure doesn't localize.

\subsection{Organization of the article and structure of the proof}

The proof of our Main Theorem follows the same blueprint that was developed in \cite{BM10} -- we will review this framework and flag for the reader the crucial differences in the present article.

Assume the setting of the Main Theorem, and consider the map $R \to R \otimes_{F[t]} \frac{F[t]}{(t - \alpha)}$ for $0 \neq \alpha \in F$. If $I_\alpha$ denotes the expansion of $I$ along this map, \cite[Proposition 1.1]{BM10} applied to this situation says that if
\begin{enumerate}
    \item $y^3 z^3 \not \in I_\alpha^*$ for infinitely many $\alpha\in F$ and \label{intro-1}
    \item $y^3 z^3 \in (S^{-1} I)^*$, \label{intro-2}
\end{enumerate}
    then $y^3 z^3 \not \in S^{-1}(I^*)$. The content of \Cref{sec:prelims,sec:noncontainment} is a proof of \ref{intro-1} for the (infinitely many) elements $\alpha$ of finite escape time $\ell_\alpha$ with respect to the family of dynamical systems to be described in \Cref{def:one-parameter-dynamical}. This noncontainment differs from the Brenner--Monsky example in that $y^3z^3\not\in I_\alpha^*$ in their ring for \emph{every} $\alpha\in F$. We rely heavily on \cite[Theorem 1.6]{BM10}, which provides that the element $xy$ (or more generally any homogeneous element of degree two in $F[x,y,z]/(\alpha z^4 + (x^2 + y z) (y^2 + x z))$) is a \textit{test element} for tight closure. This allows us to check \ref{intro-1} by showing that 
\begin{equation}
    x y (y^3 z^3)^{2^e} \not \in I_\alpha^{[2^e]} \label{eq:intro-alpha}
\end{equation}
for some $e \geq 1$, and we will pick $e = 2 \ell_\alpha - 1$ for this purpose.
    
The content of Section \ref{sec:containment} is a proof of \ref{intro-2} which is by comparison more direct, using the Hilbert--Kunz theory of Line--$S_4$ quartics established in \cite{Mon98a}. We include a discussion about the aforementioned elements of finite escape time in Section \ref{sec:density}. We conclude the article with a list of questions in Section \ref{sec:questions}.

\subsection*{Acknowledgments} We owe a significant intellectual debt to the work of Holger Brenner and Paul Monsky in \cite{BM10} and Paul Monsky in \cite{Mon98a}, without which this article would not have been written. This article was produced during the 2022 SMALL REU at Williams College --- we thank Steven J. Miller for creating a friendly and stimulating research environment, and we kindly acknowledge support through NSF grant DMS \#1947438.

\section{Notation and Constructions}\label{sec:prelims}

Throughout this section we let $F$ denote any algebraically closed field of characteristic $2$ (however, we will eventually specialize to $F=\overline{\FF_2}$). For the reader's convenience, we reproduce some of the relevant definitions from \cite{Mon98a}, using the same notation as that article whenever possible.
\begin{definition}
\label{def:one-parameter-dynamical}
To each $\alpha \in F$, we associate a map $\upvarphi_\alpha : F \cup \{ \infty \} \to F \cup \{ \infty \}$ given by $t \mapsto t^4 + \alpha t^{-4}$, denoting repeated application of this map by $\upvarphi_\alpha^{(1)} = \upvarphi_\alpha$ and $\upvarphi_\alpha^{(r+1)} = \upvarphi_\alpha \circ \upvarphi_\alpha^{(r)}$. We say that the \textit{escape time} of $\alpha$ is given by $\ell_\alpha = r$, where $\upvarphi_\alpha^{(r)}(1) = 0$. If no such $r$ exists, we say that $\ell_\alpha = \infty$. Observe that if such an $r$ does exist, it is unique because $\upvarphi_\alpha^{(r)}(0) = \infty$ for all $r > 0$.
\end{definition}

At certain points we will employ the following generalization of this dynamical system.
\begin{definition}(\emph{cf.} \cite[Definition 2.1, 2.2]{Mon98a})
\label{def:two-parameter-dynamical}
Let $2^\ZZ$ be the subgroup of $\QQ^*$ generated by $2$, and let $X = 2^\ZZ \times (F \cup \{ \infty \})$. To each $\alpha \in F$, we associate a map $\Phi_\alpha : X \to X$ given by $(Q, t) \mapsto (Q / 4, t + \alpha^Q t^{-1})$. Observe that $\alpha^Q$ is defined even for $Q < 1$ because $F$ is algebraically closed. We define repeated application $\Phi^{(r)}$ analogously to \Cref{def:one-parameter-dynamical}. We also define \textit{escape time} analogously, by saying that $L_\alpha(Q, t) = r$ if $\Phi_\alpha^{(r)}(Q, t) = (Q / 4^r, 0)$, and $L_\alpha(Q, t) = \infty$ if no such $r$ exists.
\end{definition}

Now let $\alpha$ be a nonzero element of $F$ of finite escape time, and let $Q = 2^{2\ell_a-1}$. Denote by $\hO$ the graded ring $F[x, y, z] / (x^{4Q}, y^{4Q}, z^{4Q})$, and let $v \in \hO$ be the image of $x y f^Q$ where $f=y^3z^3$. To show the noncontainment (\Cref{eq:intro-alpha}) from \Cref{sec:introduction} we aim to show that $v \notin h_\alpha \hO$.

Consider the $F$-linear map $\hO_{6Q-5} \to \hO_{6Q-1} \oplus \hO_{12Q-3}$ given by $u \mapsto (h_\alpha u, v u)$. By \cite[Theorem 4.8]{Mon98a}, there exist nonzero elements of the domain annihilated by $h_\alpha$, so in order to show that $v \notin h_\alpha \hO$, it suffices to show that this map is injective. We wish to restrict to a subspace of $\hO_{6Q-5}$ in order to simplify this verification. To that end, it is useful to introduce the notation $[i, j] = A_x^i A_y^j + A_x^j A_y^i$ where $A_x = x^2 + y z$ and $A_y = y^2 + x z$.

\begin{definition}\label{def:subspace}
We define two subspaces of $\hO_{6Q-5}$ and one subspace of $\hO_{6Q-1}$ as follows.
\begin{enumerate}
    \item (\emph{cf.} \cite[Definition 4.1]{Mon98a})
Let $W \subseteq \hO_{6Q-5}$ be spanned by $[0, Q - 2] z^{4Q-1}$ together with the generators $[i, j] z^k$, where $i$, $j$, and $k$ satisfy the following conditions: $i < j < 2Q$, $k < 4 Q$, $k \equiv 1 \pmod{4}$, and $2i + 2j + k = 6Q - 5$.\label{def:w}
\item (\emph{cf.} \cite[Definition 4.2]{Mon98a})
Let $W_0 \subseteq W$ be the subspace generated by $[0, Q - 2] z^{4Q-1}$ together with the generators $[i, j] z^k$ of \Cref{def:w} with $j\neq 2Q-1$.\label{def:w0}
\item (\emph{cf.} \cite[Definition 4.3]{Mon98a}) Let $W' \subseteq \hO_{6Q-1}$ be spanned by $[1, Q - 1] z^{4Q-1}$ together with the generators $[i, j] z^k$, where $i$, $j$, and $k$ satisfy the following conditions: $i < j < 2Q$, $k < 4 Q$, $k \equiv 1 \pmod{4}$, and $2i + 2j + k = 6Q - 1$.\label{def:w'}
\end{enumerate}
\end{definition}

In order to restrict our map $\hO_{6Q-5} \to \hO_{6Q-1} \oplus \hO_{12Q-3}$ to a map $W \to W' \oplus \hO_{12Q-3}$, we rely on \cite[Theorem 4.4]{Mon98a} which states that $h_\alpha \cdot W \subseteq W'$. Simplifying matters more, we aim to quotient by $W_0$ and prove injectivity of the resulting map. Consider the following commutative diagram with exact rows:
\begin{equation}\label{eq:diagram}
    \begin{tikzcd}[row sep=huge]
    0 \arrow[r] & W_0 \arrow[r] \arrow[d, "u \mapsto h_\alpha u" labl] & W \arrow[r] \arrow[d, "u \mapsto (h_\alpha u{,} v u)" labl] & \frac{W}{W_0} \arrow[r] \arrow[d, "u\mapsto (h_\alpha u{,} v u)" labl,"\psi"']&0\\
    0\arrow[r]&h_\alpha W_0\arrow[r,"u\mapsto (u{,}0)"]&W'\oplus\hO_{12Q-3}\arrow[r]&\frac{W'}{h_\alpha W_0}\oplus\hO_{12Q-3}\arrow[r]&0.
    \end{tikzcd}
\end{equation}
It is nontrivial that the left square commutes and the right vertical map $\psi$ is well--defined. To prove both of these facts, we must show that $v \cdot w = 0 \in \hO_{12 Q - 3}$ for all
$w \in W_0$. It suffices to check that $v$ annihilates the basis described in
\Cref{def:w0} of \Cref{def:subspace}.
It is easy to see that $v \cdot [0, Q - 2] z^{4Q-1} = 0$, because each monomial
has $z$-degree at least $7Q - 1$. Also for degree reasons, the following
is immediate:

\begin{lemma}
  For $[i, j] z^k$ satisfying the conditions in \Cref{def:w0} of \Cref{def:subspace},
  $v \cdot [i, j] z^k = c (x y z)^{4Q-1}$, where $c$ is the coefficient of
  $x^{4Q-2} y^{Q-2} z^{Q-1}$ in $[i, j] z^k$.
\end{lemma}

\subsection{Proof that \texorpdfstring{$\psi$}{psi} is well--defined}

\begin{lemma}
  \label{lem:relations-a-b}
  In order for $[i, j] z^k$ satisfying the conditions of \Cref{def:w0} of \Cref{def:subspace} to
  contain a nonzero $x^{4Q-2} y^{Q-2} z^{Q-1}$ term, it is necessary that the
  following equations admit integral solutions:
  \begin{equation*}
    a = \frac{6Q - 6 - i - 2j}{3},
    \qquad
    b = \frac{9Q - 6 - 2i - j}{3},
  \end{equation*}
  with $0 \leq a \leq i$ and $0 \leq b \leq j$. In particular, for
  $v \cdot [i, j] z^k$ to be nonzero, we must have $i \equiv j \pmod{3}$. In
  this case, moreover, we have
  \begin{equation*}
    v \cdot [i, j] z^k = \binom{i}{a} \binom{j}{b} (xyz)^{4Q-1}.
  \end{equation*}
\end{lemma}

\begin{proof}
  Recalling that $[i, j] z^k = (A_x^i A_y^j + A_x^j A_y^i) z^k$, we claim that
  the left summand $A_x^i A_y^j z^k$ does not contain any monomial term of the
  form $x^{4Q-2} y^{Q-2} z^{Q-1}$. Expanding this left summand, we have
  \begin{equation*}
    A_x^i A_y^j z^k
    = \sum_{a=0}^i \sum_{b=0}^j \binom{i}{a} \binom{j}{b} x^{2a+j-b} y^{i-a+2b} z^{i-a+j-b+k}.
  \end{equation*}
  For this sum to contain a monomial of the desired form, there must exist an
  integral solution to the system
  \begin{equation*}
    \begin{cases}
      2a + j - b = 4Q - 2 \\
      i - a + 2b = Q - 2 \\
      i - a + j - b + k = Q - 1,
    \end{cases}
  \end{equation*}
  satisfying $0 \leq a \leq i$ and $0 \leq b \leq j$. Solving for $a$ in the
  above system, we find that $i \geq a = (9Q - 6 - i - 2j) / 3$. From the
  conditions in \Cref{def:w0} of \Cref{def:subspace}, recall that $0 < k = 6Q - 5 - 2i - 2j$. One can
  easily verify that any point satisfying these two inequalities must violate
  the condition $i < j$. We conclude that the above stated system does not admit
  any solutions, and so the left summand $A_x^i A_y^j z^k$ does not contain any
  $x^{4Q-2} y^{Q-2} z^{Q-1}$ term.

  Turning our attention to the right summand $A_x^j A_y^i z^k$, we follow a
  similar argument. Expanding, we have
  \begin{equation*}
    A_x^j A_y^i z^k
    = \sum_{a=0}^i \sum_{b=0}^j \binom{i}{a} \binom{j}{b} x^{i-a+2b} y^{2a+j-b} z^{i-a+j-b+k}.
  \end{equation*}

  In order for this sum to contain the desired monomial $x^{4Q-2} y^{Q-2} z^{Q-1}$, there must exist an
  integral solution to the system
  \begin{equation*}
    \begin{cases}
      2a + j - b = 4Q - 2 \\
      i - a + 2b = Q - 2 \\
      i - a + j - b + k = Q - 1,
    \end{cases}
  \end{equation*}
  satisfying $0 \leq a \leq i$ and $0 \leq b \leq j$. Solving for $a$ and $b$,
  we obtain the desired relations.
\end{proof}

Fixing $i$ and $a$, there exist unique $j$ and $b$ solving the system of
\Cref{lem:relations-a-b} given by
\begin{equation*}
  j = \frac{6Q - 6 - 3 a - i}{2}
  \quad \textrm{and} \quad
  b = \frac{4Q - 2 + a - i}{2}.
\end{equation*}
Observe that these values are integral if and only $i \equiv a \pmod{2}$. Rewriting in terms of $i$ and $a$ the conditions in \Cref{def:w0} of \Cref{def:subspace} (namely, $j \leq 2 Q - 2$ and $k = 6Q - 5 - 2i - 2j \geq 0$), together with the condition $b \leq j$ from \Cref{lem:relations-a-b}, it is sufficient to check values of $i$ and $a$ satisfying

\begin{equation}
  \label{eq:conditions-i-a}
  2^{2\ell} - 2 - 3 a \leq i \leq 3 a + 1
  \quad \textrm{and} \quad
  a \leq 2^{2\ell-2} - 1.
\end{equation}
Therefore, because $\characteristic F = 2$, we conclude that in order to show
that $v \cdot W_0 = 0$, it suffices to show that the product of binomial
coefficients
\begin{equation*}
  \binom{i}{a} \binom{j}{b}
  = \binom{i}{a} \binom{\frac{6Q - 6 - 3a - i}{2}}{\frac{4Q - 2 + a - i}{2}}
\end{equation*}
is even for all $i \equiv a \pmod{2}$ satisfying the inequalities in
\cref{eq:conditions-i-a}.

In order to prove this claim, we rely on a tiling argument using Sierpi\'nski triangles. Broadly, we observe that the inequalities in $i$ and $a$ described in \cref{eq:conditions-i-a} cut out a triangular region $A$ in the $ia$-plane. We next describe a family of triangular regions in the $ia$-plane for which $\binom{i}{a}$ is even (these triangular regions correspond to certain subsets of Sierpi\'nski's triangle). These regions do not quite cover $A$, but if we consider the copy of Sierpi\'nski's triangle in the $jb$-plane, representing the $\binom{j}{b}$ coefficient, and we apply the appropriate coordinate transformation to map this triangle into the $ia$-plane, then by taking subsets from both copies of Sierpi\'nski's triangle together, we are able to cover $A$ (See \Cref{fig:sierp}).

\begin{figure}
    \centering
    \begin{tikzpicture}[line cap=round,line join=round,>=triangle 45,x=0.02cm,y=0.02cm]

\begin{axis}[
x=0.02cm,y=0.02cm,
axis lines=middle,
ymajorgrids=true,
xmajorgrids=true,
xmin=230,
xmax=900,
ymin=-30,
ymax=400,
xticklabels={},
yticklabels={},
y axis line style={opacity=0}
]

\clip(128,-133) rectangle (1072,590);

\fill[line width=2pt,color=color1,fill=color1,fill opacity=0.75] (511,171) -- (511,170) -- (510,170) -- cycle;
\fill[line width=2pt,color=color2,fill=color2,fill opacity=0.71] (510.5,169.5) -- (511.25,170.25) -- (508.25,170.25) -- cycle;
\fill[line width=2pt,color=gray,fill=gray,fill opacity=0.1] (511.5,170.16666666666666) -- (766,255) -- (256,255) -- cycle;
\fill[line width=2pt,color=color1,fill=color1,fill opacity=0.66] (512,1) -- (512,511) -- (1022,511) -- cycle;
\fill[line width=2pt,color=color1,fill=color1,fill opacity=0.7] (256,1) -- (256,255) -- (510,255) -- cycle;
\fill[line width=2pt,color=color1,fill=color1,fill opacity=0.72] (448,129) -- (448,191) -- (510,191) -- cycle;
\fill[line width=2pt,color=color1,fill=color1,fill opacity=0.69] (496,161) -- (496,175) -- (510,175) -- cycle;
\fill[line width=2pt,color=color1,fill=color1,fill opacity=0.71] (508,169) -- (508,171) -- (510,171) -- cycle;
\fill[line width=2pt,color=color2,fill=color2,fill opacity=0.68] (447.5,191.5) -- (510,255) -- (699.5,191.5) -- cycle;
\fill[line width=2pt,color=color2,fill=color2,fill opacity=0.63] (495.5,175.5) -- (510.5,190.5) -- (555.5,175.5) -- cycle;
\fill[line width=2pt,color=color2,fill=color2,fill opacity=0.74] (507.5,171.5) -- (510.5,174.5) -- (519.5,171.5) -- cycle;

\definecolor{gray}{RGB}{190,190,190}

\draw [line width=2pt,color=color1] (511,171)-- (511,170);
\draw [line width=2pt,color=color1] (511,170)-- (510,170);
\draw [line width=2pt,color=color1] (510,170)-- (511,171);
\draw [line width=2pt,color=color2] (510.5,169.5)-- (511.25,170.25);
\draw [line width=2pt,color=color2] (511.25,170.25)-- (508.25,170.25);
\draw [line width=2pt,color=color2] (508.25,170.25)-- (510.5,169.5);
\draw [line width=2pt,color=gray] (511.5,170.16666666666666)-- (766,255);
\draw [line width=2pt,color=gray] (766,255)-- (256,255);
\draw [line width=2pt,color=gray] (256,255)-- (511.5,170.16666666666666);
\draw [line width=2pt,color=color1] (512,1)-- (512,511);
\draw [line width=2pt,color=color1] (512,511)-- (1022,511);
\draw [line width=2pt,color=color1] (1022,511)-- (512,1);
\draw [line width=2pt,color=color1] (256,1)-- (256,255);
\draw [line width=2pt,color=color1] (256,255)-- (510,255);
\draw [line width=2pt,color=color1] (510,255)-- (256,1);
\draw [line width=2pt,color=color1] (448,129)-- (448,191);
\draw [line width=2pt,color=color1] (448,191)-- (510,191);
\draw [line width=2pt,color=color1] (510,191)-- (448,129);
\draw [line width=2pt,color=color1] (496,161)-- (496,175);
\draw [line width=2pt,color=color1] (496,175)-- (510,175);
\draw [line width=2pt,color=color1] (510,175)-- (496,161);
\draw [line width=2pt,color=color1] (508,169)-- (508,171);
\draw [line width=2pt,color=color1] (508,171)-- (510,171);
\draw [line width=2pt,color=color1] (510,171)-- (508,169);
\draw [line width=2pt,color=color2] (447.5,191.5)-- (510,255);
\draw [line width=2pt,color=color2] (510,255)-- (699.5,191.5);
\draw [line width=2pt,color=color2] (699.5,191.5)-- (447.5,191.5);
\draw [line width=2pt,color=color2] (495.5,175.5)-- (510.5,190.5);
\draw [line width=2pt,color=color2] (510.5,190.5)-- (555.5,175.5);
\draw [line width=2pt,color=color2] (555.5,175.5)-- (495.5,175.5);
\draw [line width=2pt,color=color2] (507.5,171.5)-- (510.5,174.5);
\draw [line width=2pt,color=color2] (510.5,174.5)-- (519.5,171.5);
\draw [line width=2pt,color=color2] (519.5,171.5)-- (507.5,171.5);

\draw (243,-6) node [anchor=north west][inner sep=0.75pt]    {$Q/2$};

\draw (-6,500) node [anchor=north west][inner sep=0.75pt]    {$Q/2$};

\draw (500,-6) node [anchor=north west][inner sep=0.75pt]    {$Q$};

\end{axis}
\end{tikzpicture}
    \caption{Sierpi\'nski triangles in the $ia$-plane. The red triangles are regions of the Siepri\'nski triangle representing even $\binom{i}{a}$, and the blue triangles are regions of the Sierpi\'nski triangle representing even $\binom{j}{b}$. The gray triangle bounds $A$.}
    \label{fig:sierp}
\end{figure}
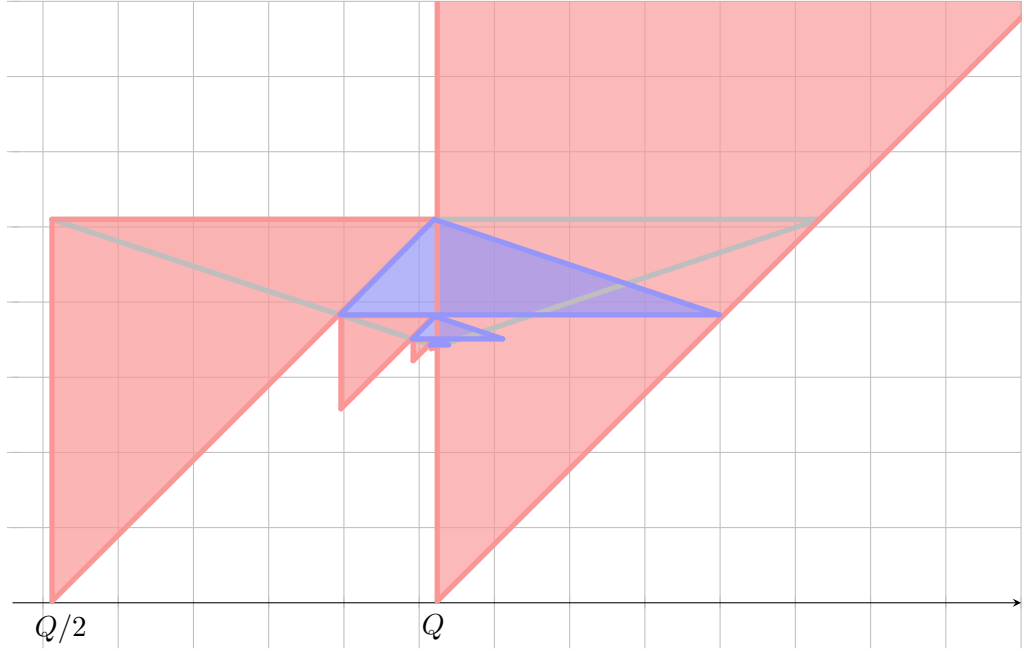

\Cref{lem:lucas,lem:triangles} together describe the triangular regions in the $ia$-plane for which $\binom{i}{a}$ is even, \Cref{lem:j-b-triangles} performs the coordinate transformation from the $jb$-coordinates to the $ia$-coordinates, and \Cref{lem:tiling} shows that our collection of subsets from the two copies of Sierpi\'nski's triangle together tile $A$.

\begin{lemma}[Lucas's theorem, \cite{Luc78}]
\label{lem:lucas}
Let $n$ and $r$ be nonnegative integers with binary expansions
$n = \sum_{i \geq 0} a_i 2^i$ and $r = \sum_{i \geq 0} b_i 2^i$. Then the
binomial coefficient $\binom{n}{r}$ is odd if and only if $b_i \leq a_i$ for all
$i \geq 0$.
\end{lemma}

\begin{lemma}
  \label{lem:triangles}
  For any integers $s, t, u, n, r \geq 0$ satisfying $s 2^t \leq n < (s + 1) 2^t - 1$ and $n + 1 - (s - u) 2^t \leq r < (u + 1) 2^t$, the binomial coefficient $\binom{n}{r}$ is even.
\end{lemma}

\begin{proof}
  In the case that $t = 0$, the statement is vacuously true. We will assume that
  $t > 0$. Observe that $n = s 2^t + m$ for some $0 \leq m < 2^t - 1$, and that
  $r = m + c$ for some $u 2^t + 1 \leq c < (u + 1) 2^t - m$. Denote the binary
  expansions of $n$ and $r$ by $n = \sum_{i\geq 0} a_i 2^i$ and
  $r = \sum_{i\geq 0} b_i 2^i$, with $0 \leq a_i, b_i \leq 1$. Note that
  $m = \sum_{i=0}^{t-1} a_i 2^i$, and so by the bounds on $m$, $a_i = 0$ for at
  least one $i < t$. Because $m < r < 2^t$, there is at least one $i$ such that
  $b_i = 1$ and $a_i = 0$. By \Cref{lem:lucas}, we conclude that $\binom{n}{r}$
  is even.
\end{proof}

\begin{lemma}
  \label{lem:j-b-triangles}
  For each $1 \leq n \leq \ell - 2$ and for all $i \equiv a \pmod{2}$ satisfying
   \begin{align*}
     i &\leq (2^{2n} + 1) 2^{2(\ell-n)} - 3 a - 6, \\
   a &< i + 2 + \frac{1}{3} \left( 1 - 2^{2n} \right) 2^{2(\ell-n)}, \\
   a &\geq \frac{1}{3} \left( 2^{2n+1} + 1 \right) 2^{2(\ell-n-1)},
  \end{align*}
  the binomial coefficient
  $\displaystyle \binom{j}{b} = \binom{\frac{6Q - 6 - 3a - i}{2}}{\frac{4Q - 2 + a - i}{2}}$
  is even. 
  The regions cut out by these inequalities correspond exactly to those regions described in \Cref{lem:triangles}, up to a coordinate transformation from the $jb$-plane to the $ia$-plane.
\end{lemma}

\begin{proof}
  Taking $j = \frac{6Q - 6 - 3a - i}{2}$ and $b = \frac{4Q - 2 + a - i}{2}$, we can solve for $i$ and $a$ in terms of $j$ and $b$. Then, substituting these expressions into the inequalities of the hypothesis, we can write these inequalities in the form
   \begin{align*}
     j &\geq \left( 2^{2n+1} - 1 \right) 2^{2(\ell-n)-1}, \\
     b &< \frac{1}{3} \left( 5 \cdot 2^{2n} + 1 \right) 2^{2(\ell-n)-1}, \\
     b &\geq j + 2 + \frac{1}{3} \left( 1 - 2^{2n} \right) 2^{2(\ell-n)-1}.
   \end{align*}
  Applying \Cref{lem:triangles} with $t = 2 (\ell - n) - 1$, $s = 2^{2n+1} - 1$,
  and $u = \frac{1}{3} \left( 5 \cdot 2^{2n} - 2 \right)$, we conclude that
  $\binom{j}{b}$ is even.
\end{proof}

\begin{lemma}
\label{lem:extra-line}
When $i = Q-1$ the binomial coefficient $\binom{j}{b}$ is even
\end{lemma}

\begin{proof}
We have that $j \leq 2Q-2$. Additionally, substituting $Q-1$ for $i$ in the equation $2i+2j \geq 5Q-5$ gives us $ j \geq \frac{3Q}{2}-1$. Now because $i \equiv 4 \mod(6)$, we can restrict these bounds to $\frac{3Q}{2}+4 \leq j \leq 2Q -6$. 

Now let $j = 2Q - m$. The above bounds give us $6 \leq m \leq \frac{Q}{2}-4$. Plugging this into $\binom{j}{b}$ we get $\displaystyle \binom{2Q-m}{\frac{5Q-4+m}{3}}$. Now because $m \leq \frac{Q}{2}$, we know that $2Q-m$ has a zero in the second most significant bit. But $\frac{5Q}{3}$  retains the one at that position, which is unchanged after adding $\frac{m-4}{3}$. We conclude from \Cref{lem:lucas} that $\binom{j}{b}$ is even.
\end{proof}

\begin{lemma}
\label{lem:tiling}

    For all $i \equiv a \pmod{2}$ with $(i,a)\in A$, the product of binomial coefficients
  $\displaystyle \binom{i}{a} \binom{j}{b} = \binom{i}{a} \binom{\frac{6Q - 6 - 3a - i}{2}}{\frac{4Q - 2 + a - i}{2}}$
  is even.
\end{lemma}

\begin{proof}
  For each $1 \leq n \leq \ell - 1$, define the sets
  \begin{gather*}
    Q_n = \left\{ (i, a)
      \ \Bigg\vert \ 
      \begin{gathered}
        i \geq \left( 2^{2n-1} - 1 \right) 2^{2(\ell-n)}, \\
        i + 1 - \frac{1}{3} \left( 2^{2n} - 1 \right) 2^{2(\ell-n)} \leq a < \frac{1}{3} \left( 2^{2n-1} + 1 \right) 2^{2(\ell-n)}
      \end{gathered}
    \right\}, \\
    Q'_n = \left\{ (i, a)
      \ \Bigg\vert \ 
      \begin{gathered}
        i \leq (2^{2n} + 1) 2^{2(\ell-n)} - 3 a - 6, \\
        \frac{1}{3} \left( 2^{2n+1} + 1 \right) 2^{2(\ell-n-1)} \leq a < i + 2 - \frac{1}{3} \left( 2^{2n} - 1 \right) 2^{2(\ell-n)}, \\
      \end{gathered}
    \right\}.
  \end{gather*}
  The first key observation is that, because the lower bound on $a$ for $Q_n$ is immediately adjacent to the upper bound on $a$ for $Q'_n$, the sets $Q_n$ and $Q'_n$ together cover a rectangle in the plane, given by
  \begin{equation*}
    Q_n \cup Q'_n \supseteq
    R_n =
    \left\{ (i, a) \ \Bigg\vert \ 
      \begin{gathered}
        \left( 2^{2n-1} - 1 \right) 2^{2(\ell-n)}
        \leq i
        \leq 2^{2\ell-1} - 2 \\
        \frac{1}{3} \left( 2^{2n+1} + 1 \right) 2^{2(\ell-n-1)}
        \leq a
        < \frac{1}{3} \left( 2^{2n-1} + 1 \right) 2^{2(\ell-n)}
      \end{gathered}
    \right\},
  \end{equation*}
  where the bounds on $i$ are inferred by arithmetic from the bounds on $Q_n$ and $Q'_n$.
  Next, we truncate each of these rectangular regions into a trapezoid as
  follows:
  \begin{equation*}
    R_n \supseteq
    R_n' =
    \left\{ (i, a) \ \Bigg\vert \ 
      \begin{gathered}
        2^{2\ell} - 2 - 3 a
        \leq i
        \leq 2^{2\ell-1} - 2 \\
        \frac{1}{3} \left( 2^{2n+1} + 1 \right) 2^{2(\ell-n-1)}
        \leq a
        \leq \frac{1}{3} \left( 2^{2n-1} + 1 \right) 2^{2(\ell-n)} - 1
      \end{gathered}
    \right\}.
  \end{equation*}
  The second observation is that these trapezoids all share a common right edge,
  and that the lower bound on $a$ for $R'_n$ is immediately adjacent to the upper bound on $a$ for $R'_{n+1}$ (not missing any
  integral values of $a$). Therefore, the union of all of the $R'_n$ contains the
  region
  \begin{equation*}
    \bigcup_{n=1}^{\ell-1} R'_n \supseteq
    T =
    \left\{ (i, a) \ \Bigg\vert \
      \begin{gathered}
        2^{2\ell} - 2 - 3 a \leq i \leq 2^{2\ell-1} - 2, \\
        \frac{1}{3} \left(2^{2\ell-1} - 1 \right) \leq a \leq 2^{2\ell-2} - 1
      \end{gathered}
    \right\}.
  \end{equation*}
  \Cref{lem:triangles,lem:j-b-triangles} tell
  us that, for each $1 \leq n \leq \ell - 1$, $\binom{i}{a}$ is even for all $(i, a) \in Q_n$ and
  $\displaystyle \binom{\frac{6Q - 6 - 3a - i}{2}}{\frac{4Q - 2 + a - i}{2}}$ is
  even for each $(i, a) \in Q'_n$. Since $T$ is covered by the $Q_n$ and the
  $Q'_n$, we conclude that the product of the binomial coefficients is even for
  all $(i, a) \in T$.

  By \Cref{lem:extra-line}, we have that the binomial coefficient
  $\displaystyle \binom{\frac{6Q - 6 - 3a - i}{2}}{\frac{4Q - 2 + a - i}{2}}$ is
  even when $i = 2^{2\ell-1} - 1$ and
  $\frac{1}{6} \left( 2^{2\ell} - 3 \right) \leq a \leq 2^{2\ell-2} - 1$.

  Finally, by \Cref{lem:triangles}, the binomial coefficient $\binom{i}{a}$ is
  even for $i \equiv a \pmod{2}$ satisfying
  $2^{2\ell-1} \leq i \leq 2^{2\ell}-2$ and
  $i + 1 - 2^{2\ell-1} \leq a < 2^{2\ell-1}$. It follows that
  $\binom{i}{a}$ is even for all $i$ and $a$ in the region
  \begin{equation*}
    T' = \left\{ (i, a) \ \Bigg\vert \ 
      \begin{gathered}
        2^{2\ell-1} \leq i \leq 3a + 1 \\
        a \leq 2^{2\ell-2} - 1
      \end{gathered}
    \right\}.
  \end{equation*}

  Observe that every $i \equiv a \pmod{2}$ satisfying the conditions
  in \cref{eq:conditions-i-a} is contained in either $T$, $T'$, or the line
  $i = 2^{2\ell-2} - 1$. We conclude that the product of binomial coefficients
  $\displaystyle \binom{i}{a} \binom{\frac{6Q - 6 - 3a - i}{2}}{\frac{4Q - 2 + a
      - i}{2}}$ is even.
\end{proof}

The next result immediately follows from \Cref{lem:tiling} and the preceding discussion.

\begin{theorem}
  The multiplication by $v$ map $\hO_{6Q-5} \to \hO_{12Q-3}$ induces a well--defined map on the quotient $W / W_0 \to \hO_{12Q-3}$. Hence, the map $\psi$ in (\ref{eq:diagram}) is well--defined.
\end{theorem}

\section{Tight closure for elements of finite escape time}\label{sec:noncontainment}

In this section we show that $f=y^3z^3\not\in I_\alpha^*$ for every $\alpha\in \overline{\FF_2}$ with finite escape time. It suffices to show that $\psi$ is injective. We first define a more general map $T$ and prove injectivity for this map.

Throughout, let $\alpha, t$ be elements of $\overline{\FF_2}$ and $Q=4^\ell/2$ where $\ell:=\ell_\alpha$.

\subsection{Matrix of the \texorpdfstring{$\psi$}{psi} map}\label{subsec:matrix}

By \cite[Theorem 4.6]{Mon98a} the $u_i \in W$ with $i$ odd and $1 \leq i \leq Q$ where $$u_i = [Q - i - 1, 2Q - 1] \cdot z^{2i - 1}$$ form a basis of the quotient $W / W_0$. We now show that $v \cdot u_i=0 \in \hO_{12Q-3}$ for all $i \neq 0$.

\begin{lemma}
\label{lem:v-u1}
We claim that $v \cdot u_1 = (xyz)^{4Q-1}$.
\end{lemma}

\begin{proof} 
Since $\hO_{12Q - 3}$ is a one-dimensional $F$-vector space (by pigeonhole) we have $v \cdot u_i = c (xyz)^{4Q-1}$, where $c$ is the coefficient of the monomial $x^{4Q-2} y^{Q-2} z^{Q-1}$ in $u_i$.

We have  
\begin{align*}
    u_1 =\left( (x^2+yz)^{Q-2} (y^2+xz)^{2Q-1} + (x^2+yz)^{2Q-1} (y^2+xz)^{Q-2} \right) \cdot z.
\end{align*}

Consider the expansion of $(y^2 + xz)^{2Q-1}$. For any monomial in this expansion, the sum of the $y$- and the $z$- degree is at least $2Q - 1$, while in the desired monomial $x^{4Q-2} y^{Q-2} z^{Q-1}$ the sum of the $y$- and $z$- degree is only $2 Q - 3$. Thus, $(x^2+yz)^{Q-2} (y^2+xz)^{2Q-1}=0$. Now we expand, $$(x^2 + yz)^{2Q-1} (y^2 + xz)^{Q-2} z = \sum_{a=0}^{2Q-1} \sum_{b=0}^{Q-2} \binom{2Q-1}{a} \binom{Q-2}{b} x^{Q-2+2a-b} y^{2Q-1-a+2b} z^{3Q-2-a-b}.$$ Since only $x^{4Q-2} y^{Q-2} z^{Q-1}\neq 0$,  we must have
$$a = \frac{5Q - 1}{3}, \quad b = \frac{Q - 2}{3}.$$

So when $\binom{2Q-1}{a}, \binom{Q-2}{b}$ are each odd then we have $(x^2+yz)^{2Q-1} (y^2+xz)^{Q-2}\neq0$. Since $Q$ is a power of 2, we can write $2 Q - 1$ in binary as a sequence of only 1's. Then, by \Cref{lem:lucas}, $\binom{2Q-1}{a}$ is odd for any $a$. Now consider $\binom{Q-2}{b}$. Here we have that $Q - 2$ is a sequence of all 1's ending in 0 at the $2^0$ position. By \Cref{lem:lucas}, since $b$ is even, $\binom{Q-2}{b}$ is odd. We may conclude that $v\cdot u_1 = (xyz)^{4Q-1},$ as desired.
\end{proof}

\begin{lemma}
\label{lem:v-ui}
For $i > 1$, we have $v \cdot u_i = 0$.
\end{lemma}

\begin{proof}
As in the proof of \Cref{lem:v-u1}, $v \cdot u_i = c (xyz)^{4Q-1}$, where $c$ is the coefficient of $x^{4Q-2} y^{Q-2} z^{Q-1}$ in $u_i$. As before, we may focus only on the $A_x^{2Q-1} A_y^{Q-i-1} z^{2i-1}$ term and, expanding, we obtain
\begin{multline*}
    (x^2 + yz)^{2Q-1} (y^2 + xz)^{Q-i-1} z^{2i-1} \\
    =\sum_{a=0}^{2Q-1} \sum_{b=0}^{Q-i-1} \binom{2Q-1}{a} \binom{Q-i-1}{b}
    x^{Q-i-1+2a-b} y^{2Q-1-a+2b} z^{3Q+i-3-a-b}.
\end{multline*} This forces
\begin{equation*}
    a = \frac{5Q + 2i - 3}{3},
    \quad
    b = \frac{Q + i - 3}{3}.
\end{equation*}
Observe that, because $Q = 2^{2\ell - 1} \equiv 2 \pmod{3}$, the system admits integral solutions only when $i \equiv 1 \pmod{3}$. Thus, because $i$ is odd we have that $i \equiv 1 \pmod{6}$

Now, in order to show that $v \cdot u_i = 0$, we must show that one of the binomial coefficients $\binom{2Q-1}{a}$ and $\binom{Q-i-1}{b}$ is even. As observed in the proof of \Cref{lem:v-u1}, the first of these is always odd; therefore, we must show that $\binom{Q-i-1}{b}$ is even.

Since $i \equiv 1 \pmod{6}$ it is either 1 or $7 \pmod{12}$. We have $Q - 1 = \sum_{r=0}^na_r2^r$ where $a_r=1$ for all $r$. Let $k\in \NN$ be the largest power of 2 in the binary expansion of $i$. Then in the binary expansion of $Q-1-i$ we have $a'_k=0$.

Now we consider the binary expansion of $b$. We want to show that $c_k=1$, which by \Cref{lem:lucas} implies that the binomial coefficient is even. We now introduce two claims to show that this binomial coefficient is indeed even.\\

\begin{claim}
\label{claim:sum-binary}
Suppose that $i = 1 \pmod{6}$, and denote its binary expansion by $\sum_{r=0}^k b_r 2^r$, with $b_r \in \{ 0, 1 \}$ and $b_1 = b_k = 1$. Then denoting the binary expansion of $Q - 3 + i$ by $\sum_{r \geq 0} c_r 2^r$ with $c_r \in \{ 0, 1 \}$, we have $c_k = 1$.
\end{claim}

\begin{subproof}[Proof of \Cref{claim:sum-binary}]
Because $Q$ is a power of 2, denoting the binary expansion of $Q - 3$ by $\sum_{r=0}^n a_r 2^r$ where $a_r \in \{ 0, 1 \}$, we have $a_1 = 0$ and $a_r = 1$ for all $r \neq 1$. Observe that $c_k = 0$ only if $b_r=0$ for all $1<r<k$. But since $i \equiv 1$ or $7 \pmod{12}$ we necessarily have $b_r=1$ for some $1<r<k$.
\end{subproof}

\begin{claim}
\label{claim:quotient-binary}
Suppose $s$ is divisible by 3, and denote its binary expansion by $\sum_{r=0}^n a_r 2^r$, with $a_r \in \{ 0, 1 \}$. Suppose that $a_n = 1$, and that for some $k$, $a_k = 1$ and $a_r = 0$ for all $k < r < n$. Then, denoting the binary expansion of $s / 3$ by $\sum_{r=0}^n b_r 2^r$ we have $b_k = 1$.
\end{claim}

\begin{subproof}[Proof of \Cref{claim:quotient-binary}]
Let $t = \sum_{r=0}^{n-k} a_{r+k} 2^r$. We consider two cases. First, suppose that $n - k$ is odd. In this case, because $t$ is one more than an odd power of 2, $t$ is divisible by 3. Because $t$ is odd, $t / 3$ is also odd, so denoting the binary expansion of $t / 3$ by $\sum_{r \geq 0} c_r 2^r$ with $c_r \in \{ 0, 1 \}$, we have that $c_0 = 1$. Observe that $s$ can be written in the form $s = t \cdot 2^k + s'$, where $s' < 2^k$. Therefore, we have
\begin{equation*}
    \frac{s}{3}
    = \frac{t}{3} \cdot 2^k + \frac{s'}{3}
    = \sum_{r \geq k} c_{r-k} 2^r + \frac{s'}{3},
\end{equation*}
with $s' / 3 < 2^k$. We conclude that $b_k = c_0 = 1$.

Next consider the case in which $n - k$ is even. In this case, because $t$ is one more than an even power of 2, we have $t \equiv 2 \pmod{3}$. Also, because $t$ is odd, $(t - 2) / 3$ is also odd, so expressing $t / 3$ in the form $\sum_{r \geq 0} c_r 2^r + 2 / 3$ with $c_r \in \{ 0, 1 \}$, we have $c_0 = 1$. Observe that $s$ can be written in the form $s = t \cdot 2^k + s'$, where $s' < 2^k$. Therefore, we have
\begin{equation*}
    \frac{s}{3}
    = \frac{t}{3} \cdot 2^k + \frac{s'}{3}
    = \sum_{r \geq k} c_{r-k} 2^r + \frac{2}{3} \cdot 2^k + \frac{s'}{3}.
\end{equation*}
Since $s' < 2^k$, we have $(2 / 3) 2^k + (s' / 3) < 2^k$, and hence we conclude that $b_k = c_0 = 1$.
\end{subproof}

By \Cref{claim:sum-binary}, we have that the $k\nth$ bit in $Q - 3 + i$ is 1. We have $Q - 3 + i = \sum_{r=0}^n a_r 2^r$, where $a_n = 1$, $a_k = 1$, and $a_r = 0$ for $k < r < n$. Now consider $b = \sum_{r=0}^n b_r 2^r$. Then by \Cref{claim:quotient-binary}, we have $b_k = 1$, as desired. Thus $\binom{Q-i-1}{b}$ is always even and we are done.
\end{proof}

\subsection{Injectivity of the \texorpdfstring{$\psi$}{psi} map}
The ideas of this section are based on those appearing in \cite[Section 2]{Mon98a} modified to account for the additional basis element.

\begin{definition}
Let $U$ and $U'$ be vector spaces over $\overline{\FF_2}$ with bases $\{u_i\}_{i\in I}, \{v_i, w\}_{i\in I}$ respectively where $I=\{i\st i \textrm{  is odd and  } 1\leq i\leq Q-1\}$. We define the linear map $T: U\to U'$ with $T(u_i)=\alpha^{Q-i} v_i + t \smashoperator{\sum_{i+j =Q/4^s}} v_j + \delta_{i,1} w$. Let $N_\alpha^+(Q,t)$ denote the nullity of the $T$ map.
\end{definition}

Suppose $v_1,\ldots v_{Q-1}$ are elements of a vector space over $\overline{\FF_2}$ and for $1\leq j \leq Q-1$ we have $F_j=\smashoperator{\sum_{i+(Q-j) =Q/4^s}} v_i$, $s\in\ZZ,s\geq 0$. Then by Lemma 2.5 of \cite{Mon98a} the same formula holds if we interchange the $v_i$ and the $F_i$. In particular, $v_{Q-i}=\smashoperator{\sum_{i + j = Q / 4^s}} F_j$.

\begin{lemma}
\label{lem:same-nullity}
The following linear maps $U \to U'$ have identical nullity:
\begin{gather*}
    a :u_i \mapsto tv_i + \alpha^i \smashoperator{\sum_{i+j=Q/4^s}} v_j + \delta_{i,1} w, \\
    b : u_i \mapsto \alpha^{Q-i} v_i + t \smashoperator{\sum_{i+j =Q/4^s}} v_j + \delta_{i,1} w.
\end{gather*}
\end{lemma}

\begin{proof}
Let $E_i = u_{Q-i}$. Then  $\displaystyle a : E_i \mapsto \alpha^{Q-i} F_i +  t \smashoperator{\sum_{i+j =Q/4^s}} F_j + \delta_{Q-i,1} w$. Notice that in the $E_i$ basis $\delta_{Q-i,1} w = 0$ for all $i\neq Q-1$. Hence $\ker a\subseteq\Span \{ E_i \st 1 \leq i < Q - 1 \}$. Now since $i\neq Q-1$ in the kernel, we know that the $\alpha^{Q-i}F_i$ term in $a(E_i)$ cannot contribute an $F_{Q-1}$ term. We only get an $F_{Q-1}$ from $ \smashoperator{\sum_{i+j =Q/4^s}} F_j$ when $i=1$ so $\ker a\subseteq\Span \{ E_i \st 1 < i < Q - 1 \}$. The same argument can be repeated for the $b$ map, so that $\ker b\subseteq\Span \{ u_i \st 1 < i < Q - 1 \}$. Then by \cite[Lemma 2.6]{Mon98a} the maps
\begin{equation*}
    c: u_i \mapsto tv_i + \alpha^i \smashoperator{\sum_{i+j=Q/4^s}} v_j \hspace{20pt} \textrm{and} \hspace{20pt} d: u_i \mapsto \alpha^{Q-i} v_i + t \smashoperator{\sum_{i+j =Q/4^s}} v_j
\end{equation*}
$c$ and $d$ have the same nullity. By the above discussion, $a$ and $b$ have the same nullity too.

\end{proof}

\begin{lemma}\label{nullities are the same}
Given $Q \geq 4$ and $t \neq 0$, let $t^* = t + \alpha^Q t^{-1} \neq 0$. Then $N_\alpha^+(Q, t) = N_\alpha^+(Q / 4, t^*)$. In other words, applying $\Phi_\alpha$ does not change the nullity of $T$.
\end{lemma}

\begin{proof}
For each $i < Q / 2$ we define $G_i = \alpha^i t^{-1} u_i + u_{Q-i}$. Also, with $F_i$ defined as in \cite[Lemma 2.5]{Mon98a}, for $i < Q / 2$ we have $F_i = v_i$. Observe that, for $i < Q / 2$, we have
\begin{equation*}
    T(u_i) = \alpha^{Q-i} F_i + t v_{Q-i} + t \smashoperator{\sum_{\substack{i + j = Q / 4^s \\ s > 0}}} F_j + \delta_{i,1} w,
\end{equation*}
and similarly, $T(u_{Q-i}) = \alpha^i v_{Q-i} + t F_i$. It follows that
\begin{equation*}
T(G_i) = t^{-1}\alpha^iT( u_i) + T(u_{Q-i}) = t^* F_i + \alpha^i \sum F_j + \alpha t^{-1} \delta_{i,1} w,
\end{equation*} with the sum exactly as before. We then obtain the following formulae:
\begin{align}
\label{eq:t-e1}
T(G_1) &= t^* F_i + \alpha^i \smashoperator{\sum_{\substack{1+j=Q/4^s \\ s>0}}} F_j + \alpha t^{-1} w, & i &= 1;\\
\label{eq:t-ei}
T(G_i) &= t^* F_i + \alpha^i \smashoperator{\sum_{\substack{i+j=Q/4^s \\ s>0}}} F_j, & 1 < &i < Q / 4;& \\
T(G_i) &= t^* F_i, & Q / 4 \leq &i < Q / 2;& \\
T(u_i) &= \alpha^{Q-i} v_i + t F_{Q-i}, & Q / 2 < &i < Q.&
\end{align}
We may now show that $N_\alpha^+(Q, t) = N_\alpha^+(Q / 4, t^*)$. First we will assume that $\alpha = 0$. Because $t^* = t + \alpha^Q t^{-1}$, $t^* = t$, and it is immediate that $N_\alpha^+(Q, t) = N_\alpha^+(Q/4, t) = 0$. Now consider the case of $\alpha \neq 0$. In this case, the set $\{ G_i \st 1 \leq i < Q / 2 \} \cup \{ u_i \st Q / 2 < i \leq Q - 1 \}$ forms a basis of $U$. Since $t^* \neq 0$, $\ker T\subseteq\Span \{ G_i \st 1 \leq i < Q / 4 \}$. \Cref{eq:t-e1,eq:t-ei} together with \Cref{lem:same-nullity} show that $N_\alpha^+(Q, t) = N_\alpha^+(Q/4, t^*)$.
\end{proof}

\begin{lemma}
Given $t\neq0$, $\ell=\ell_\alpha(Q,t)<\infty$, $Q=4^\ell/2$, then $N_\alpha^+(Q,t)=0$. 
\end{lemma}
\begin{proof}
We will show this by induction. When $\ell = 1, Q = 2$ so $\alpha=1$. Thus we can check that the $2\times 1$ matrix has nullity 0. Suppose now that $\ell > 1$ with $Q \geq 4$ and $t^*$ non-zero. We may define an $\ell^*$ as $\ell_{\alpha}(Q/4,t^*)$. Suppose for induction that $Q=4^\ell/2$. Then because $\ell^*=\ell-1$, $Q/4=4^{\ell^*}/2$. Repeating this process we obtain that $N_{\alpha}^+(Q/4,t^*)=0$. Since $N_{\alpha}^+(Q/4,t^*)=0$, by \Cref{nullities are the same} $N_\alpha^+(Q,t)=0$ as desired.
\end{proof}

\begin{corollary}\label{cor:psi-inj}
The map $\psi: W/W_0 \to W'/hW_0 \oplus \hO_{12Q-3}$ is injective. 
\end{corollary}

\begin{proof}
The $\psi$ map is a particular case of $T$ where we specify the basis and fix $\alpha$ and hence $\ell$.
\end{proof}

\begin{corollary}\label{cor:noncontainment}
Let $v=(xy)f^Q$. Then $v\not\in h_\alpha\hO$. 
\end{corollary}

\begin{proof}
Suppose that $v\in h_\alpha\hO$. Then $v=h\cdot g$ for some $g\in \hO$. By \cite[Theorem 4.8]{Mon98a} we know there exists $u\neq 0\in \hO$ such that $hu=0\in \hO$. Then $(hu,vu)=(hu,ghu)=(0,0)$. However we just showed in \Cref{cor:psi-inj} that $\psi$ is injective so this is a contradiction. Thus $v\not\in h_\alpha\hO$.
\end{proof}

\begin{definition}
We say that an element $r\in R$ is a \emph{test element for tight closure} if for all ideals $I$ and elements $f$ we have $f\in I^*$ if and only if $rf^{Q}\in I^{[Q]}$ for all $Q = p^e$ (see \cite[Chapter 2]{Hun96} for more details about test elements).
\end{definition}

\begin{theorem}\label{thm:noncontainment-tight}
There exist infinitely many $\alpha\in \overline{\FF_2}$ for which $f\not\in I^*$ in $R_\alpha$.
\end{theorem}

\begin{proof}
Let $0\neq \alpha\in \overline{\FF_2}$ such that $\ell_\alpha<\infty$. \cite[Theorem 5.10]{Mon98a} guarantees that there are infinitely many such $\alpha$ (in fact, that there are $8^{n-1}$ elements of escape time $n$ for each $n$). We know $v=(xy)f^Q \not\in h_\alpha\hO$. Then $v\neq 0 \in \hO$ so $(xy)f^Q\not\in I^{[Q]}$. By \cite[Theorem 1.6]{BM10} with $p=2$ and $\deg(h)=4$ we have that every nonzero element of degree at least two is a test element for tight closure. Hence $xy$ is a test element so it follows that $f\not\in I^*$ by \Cref{cor:noncontainment}.
\end{proof}

\section{Elements of infinite escape time \texorpdfstring{and $f\in (S^{-1}I)^*$}{}}
\label{sec:containment}
This section may be viewed as a modification of \cite[Section 2]{BM10} with $g$ replaced by $h$. We first recall the notion of the \emph{Hilbert--Kunz function} at the level of generality that it will be used.
\begin{definition}\label{def:ehk-line}
Let $L\supseteq K$ where $K$ is any algebraically closed field of characteristic $2$ and let $$R_\alpha=\frac{L[x,y,z]}{(h_\alpha:=\alpha z^4 + (x^2 + y z) (y^2 + x z))}$$ where $0\neq\alpha\in L$. The \emph{Hilbert--Kunz function of $R_\alpha$} is given by $$n\mapsto e_n(R_\alpha):=\dim_L \frac{L[x,y,z]}{(x^{2^n},y^{2^n},z^{2^n},h_\alpha)}.$$
\end{definition}

\begin{definition}\label{def:A}
Let $$R = \frac{\overline{\FF_2}[x, y, z, t]}{(h_t:=tx^2y^2+z^4+xyz^2+x^3z+y^3z)},$$ $I = (x^4, y^4, z^4)$, $f = y^3 z^3$, and $S = \overline{\FF_2}[t] \setminus \{ 0 \}$.
\end{definition}

We combine the contents of this section and the previous one to obtain the Main Theorem; the proof follows \emph{mutatis mutandis} as \cite[Theorem 2.5]{BM10}.

\begin{theorem}\label{thm-main-thm}
$f$ is in $(S^{-1}I)^*$ but $f\not\in S^{-1}(I^*)$, so tight closure of $I$ doesn't commute with localization at $S$.
\end{theorem}

\begin{proof}
 One observes that enlarging the coefficient field to an algebraic closure does not affect $e_n(-)$. Monsky demonstrated in \cite[Theorem 5.10]{Mon98a} that all elements of finite escape time are algebraic over $\FF_2$. It follows that in the field $\overline{\FF_2(t)}$ the element $t$ has $\ell_t=\infty$ (recall that we only define escape time in algebraically closed fields). It then follows from \cite[Corollary 3.15]{Mon98a} that 
\begin{align}
    e_n(R_t)=3\cdot 4^n-4\label{ehk:line-s4}
\end{align}
for all $n\geq 1$. We identify $S^{-1}R$ with $\overline{\FF_2}(t)[x,y,z]/(h_t)=R_t$, which is trivial in degrees $6Q+1$ and larger by a combination\footnote{See the proof of \cite[Theorem 2.3]{BM10} for more details} of \Cref{ehk:line-s4}, \cite[Lemma 4.1]{Mon98b}, and graded artinian Gorenstein duality \cite[\textsection 21.2]{Eis13}. Since $yf^Q=y^{3Q+1}z^{3Q}$ has degree $6Q+1$, it follows that (the image of) $yf^Q=0$ in $\hO/(h_\alpha)$. This means that in $R_t/(x^{4Q},y^{4Q},z^{4Q})$, $yf^Q$ is also trivial and thus $yf^Q\in(x^{4Q}, y^{4Q}, z^{4Q}) = (x^{4}, y^{4}, z^{4})^{[Q]}$ for all $Q$. So $f\in (S^{-1}I)^*.$ The second statement follows from \Cref{thm:noncontainment-tight} and \cite[Proposition 1.1]{BM10}.
\end{proof}

\section{Density of finite escape time elements}\label{sec:density}
We include in this section a digression about elements $\alpha\in F=\overline{\FF_2}$ for which $\ell_\alpha<\infty$. Due to their seemingly irregular distribution throughout $F$, we sought a list of such examples to gain an intuition for the dynamical systems $\upvarphi_\alpha$.

\Cref{table} at the end of this article lists representative elements of finite escape time in the dynamical system $\upvarphi$ for escape time up to $\ell=427$ (with significant gaps). These are all of the escape times that the authors encountered with available computing resources. The $\beta$ have algebraic degree less than or equal to $15$, with the exception of the elements of escape time $5$ and $7$ which live in $\FF_{2^{74}}$ and $\FF_{2^{20}}$, respectively. Most of these were found via exhaustive search in Macaulay2 \cite{M2}, again with the exception of $\ell_\alpha=5,7$ where instead we take the following approach.

Let $G_n,H_n\in\FF_2[w]$ be defined inductively as
\begin{enumerate}
    \item $G_1=w+1, H_1=1$,
    \item $G_{n+1}=G_n^8+wH_n^8$, $H_{n+1}=G_n^4H_n^4$.
\end{enumerate}
Monsky observes in \cite[Theorem 5.10]{Mon98a} that after enlarging the coefficient field from $\FF_2$ to an algebraic closure, $\ell_\alpha=n$ if and only if $G_n(\alpha)=0$, where $G_n$ is a polynomial of degree $8^{n-1}$. To find elements with small escape time $n$, one may factor $G_n$ in $\FF_2[w]$, find the smallest degree irreducible factor (say $T(w)$ of degree $d$), and then pick a root of $T$ in $\FF_{2^d}$ to obtain the desired element. The following table provides a list of the degrees of the factors of $G_n$ for $1\leq n\leq 7$.

\begin{table}[H]
    \centering
    \begin{tabular}{l|l}
        $n$ & degrees of irreducible factors of $G_n$ \\\hline
        $1$ & $1$\\
        $2$ & $2,6$\\
        $3$ & $3,7,13,41$\\
        $4$ & $5,12,42,112,121,220$\\
        $5$ & $74, 4022$\\
        $6$ & $11,15,45,143,229,515,708,1704,3146,26252$\\
        $7$ & $20,76,1544,1640,84207,174657$
    \end{tabular}
    \caption{}
\end{table}

\noindent Factoring $G_8$ exceeded our computational limitations. We write one representative explicitly for $\ell_\beta=5$. We have $\beta=\gamma\in \FF_{2^{74}}$ where

\begin{align}
    \gamma:=\alpha^{68}+\alpha^{58}+\alpha^{56}+\alpha^{55}+\alpha^{54}+\alpha^{53}+\alpha^{52}+\alpha^{50}+\alpha^{47}+\alpha^{43}+\alpha^{41}+\alpha^{40}+\alpha^{36}+\alpha^{35}\label{gamma}\\
    +\alpha^{34}+\alpha^{33}+\alpha^{31}+\alpha^{29}+\alpha^{28}+\alpha
      ^{27}+\alpha^{26}+\alpha^{22}+\alpha^{20}+\alpha^{19}+\alpha^{18}+\alpha^{17}+\alpha^{3}+1\nonumber,
\end{align} where the relation defining the field $\FF_{2^{74}}$ is

\begin{align}
    \alpha^{74}=\alpha^{37}+\alpha^{36}+\alpha^{35}+\alpha^{34}+\alpha^{33}+\alpha^{32}+\alpha^{31}+\alpha^{30}+\alpha^{29}\label{l5relation}\\
    +\alpha^{28}+\alpha^{27}+\alpha^{26}+\alpha^{24}+\alpha^{21}+\alpha^{17}+\alpha^{16}+\alpha^{13}+\alpha^{12}+\alpha^{11}+\alpha^{8
      }+\alpha^{3}+1\nonumber.
\end{align}

\section{Questions}\label{sec:questions}
We conclude the article by mentioning two avenues worth exploring.
\subsection{Principal localizations}
The multiplicative set at which we localize in our example is the same as the one in \cite{BM10}, so in particular our work does not address the question (still open as of the present article) of whether the operations $(-)^*$ and $S^{-1}(-)$ commute when $S=\{1,f,f^2,\dots\}$. It was shown in \cite{Eps22} that this question is tantamount to asking whether the tight closure of a coherent ideal sheaf remains coherent.

\subsection{Prime characteristic \texorpdfstring{$p>2$}{greater than two}}
To the best of the authors' knowledge, it has not been proven as of the present article that tight closure may fail to commute with localization in rings of prime characteristic $p>2$. As mentioned in \cite[Remark 4.4]{BM10} and corroborated in \cite[Question 5.3]{Smi20}, there are rings in characteristic $3$ which appear natural to consider due to their parameter--dependent Hilbert--Kunz theories. However, new ideas are needed to circumvent the degree bound of \cite[Theorem 1.6]{BM10} and produce test elements in such rings.

\printbibliography

\newpage

\begin{table}[H]
    \begin{center}
    \resizebox{.72\linewidth}{!}{
    \begin{tabular}{l|l|r}
$\beta$ & Relation & $\ell_\beta$ \\ \hline
$0$ & -- & $\infty$ \\
$1$ & -- & 1 \\
$\alpha$ & $\alpha^{2}=\alpha+1$ & 2 \\
$\alpha+1$ & $\alpha^{3}=\alpha+1$ & 3 \\
$\alpha+1$ & $\alpha^{5}=\alpha^{2}+1$ & 4 \\
$\gamma\in\mathbb{F}_{2^{74}}$ \hfill(\ref{gamma})&(\ref{l5relation}) & 5  \\
$\alpha^{7}+\alpha^{6}+\alpha^{5}+\alpha^{4}+\alpha^{3}+1$ & $\alpha^{11}=\alpha^{2}+1$ & 6 \\
$\alpha^3+\alpha\in\mathbb{F}_{2^{20}}$ & $\alpha^{20}=\alpha^{10}+\alpha^{9}+\alpha^{7}+\alpha^{6}+^{5}+\alpha^{4}+\alpha+1$& 7\\
$\alpha^{5}+\alpha^{4}+\alpha^{3}+\alpha^{2}$ & $\alpha^{10}=\alpha^{6}+\alpha^{5}+\alpha^{3}+\alpha^{2}+\alpha+1$ & 10 \\
$\alpha^{4}+\alpha^{3}+\alpha+1$ & $\alpha^{7}=\alpha+1$ & 12 \\
$\alpha^{5}+\alpha^{2}+1$ & $\alpha^{13}=\alpha^{4}+\alpha^{3}+\alpha+1$ & 13 \\
$\alpha^{7}+\alpha^{5}+\alpha^{2}+1$ & $\alpha^{13}=\alpha^{4}+\alpha^{3}+\alpha+1$ & 17 \\
$\alpha^{6}+\alpha^{5}+\alpha+1$ & $\alpha^{10}=\alpha^{6}+\alpha^{5}+\alpha^{3}+\alpha^{2}+\alpha+1$ & 21 \\
$\alpha^{7}+\alpha^{5}+\alpha^{4}+\alpha^{2}+1$ & $\alpha^{12}=\alpha^{7}+\alpha^{6}+\alpha^{5}+\alpha^{3}+\alpha+1$ & 23 \\
$\alpha^{4}+\alpha^{3}+1$ & $\alpha^{10}=\alpha^{6}+\alpha^{5}+\alpha^{3}+\alpha^{2}+\alpha+1$ & 25 \\
$\alpha^{7}+\alpha^{6}+\alpha^{5}+\alpha^{3}+\alpha^{2}+\alpha$ & $\alpha^{13}=\alpha^{4}+\alpha^{3}+\alpha+1$ & 27 \\
$\alpha^{10}+\alpha^{5}+\alpha^{4}+\alpha^{2}$ & $\alpha^{13}=\alpha^{4}+\alpha^{3}+\alpha+1$ & 28 \\
$\alpha^{9}+\alpha^{8}+\alpha^{7}+\alpha^{6}+\alpha^{2}+1$ & $\alpha^{12}=\alpha^{7}+\alpha^{6}+\alpha^{5}+\alpha^{3}+\alpha+1$ & 29 \\
$\alpha^{6}+\alpha^{5}+\alpha^{4}$ & $\alpha^{10}=\alpha^{6}+\alpha^{5}+\alpha^{3}+\alpha^{2}+\alpha+1$ & 31 \\
$\alpha^{10}+\alpha^{9}+\alpha^{4}+\alpha^{2}$ & $\alpha^{15}=\alpha^{5}+\alpha^{4}+\alpha^{2}+1$ & 33 \\
$\alpha^{8}+\alpha^{3}+1$ & $\alpha^{14}=\alpha^{7}+\alpha^{5}+\alpha^{3}+1$ & 34 \\
$\alpha^{10}+\alpha^{8}+\alpha^{4}+\alpha^{3}+\alpha^{2}+1$ & $\alpha^{15}=\alpha^{5}+\alpha^{4}+\alpha^{2}+1$ & 35 \\
$\alpha^{4}+\alpha^{3}+\alpha^{2}+1$ & $\alpha^{12}=\alpha^{7}+\alpha^{6}+\alpha^{5}+\alpha^{3}+\alpha+1$ & 37 \\
$\alpha^{10}+\alpha^{8}+\alpha^{6}+\alpha^{5}+\alpha^{4}+\alpha^{3}+\alpha^{2}+\alpha+1$ & $\alpha^{13}=\alpha^{4}+\alpha^{3}+\alpha+1$ & 39 \\
$\alpha^{9}+\alpha^{8}+\alpha^{7}+\alpha^{3}+\alpha^{2}+\alpha+1$ & $\alpha^{14}=\alpha^{7}+\alpha^{5}+\alpha^{3}+1$ & 40 \\
$\alpha^{8}+\alpha^{6}+\alpha$ & $\alpha^{14}=\alpha^{7}+\alpha^{5}+\alpha^{3}+1$ & 43 \\
$\alpha^{11}+\alpha^{9}+\alpha^{6}+\alpha^{5}+\alpha^{4}+\alpha^{2}+1$ & $\alpha^{15}=\alpha^{5}+\alpha^{4}+\alpha^{2}+1$ & 46 \\
$\alpha^{10}+\alpha^{9}+\alpha^{8}+\alpha^{7}+\alpha^{5}+\alpha^{4}+\alpha+1$ & $\alpha^{14}=\alpha^{7}+\alpha^{5}+\alpha^{3}+1$ & 47 \\
$\alpha^{3}$ & $\alpha^{10}=\alpha^{6}+\alpha^{5}+\alpha^{3}+\alpha^{2}+\alpha+1$ & 49 \\
$\alpha^{11}+\alpha^{9}+\alpha^{7}+\alpha^{4}+\alpha^{3}$ & $\alpha^{14}=\alpha^{7}+\alpha^{5}+\alpha^{3}+1$ & 53 \\
$\alpha^{8}+\alpha^{6}+\alpha^{4}+\alpha^{3}+1$ & $\alpha^{15}=\alpha^{5}+\alpha^{4}+\alpha^{2}+1$ & 54 \\
$\alpha^{7}+\alpha^{5}+\alpha^{4}+\alpha^{3}$ & $\alpha^{12}=\alpha^{7}+\alpha^{6}+\alpha^{5}+\alpha^{3}+\alpha+1$ & 59 \\
$\alpha^{10}+\alpha^{6}+\alpha+1$ & $\alpha^{14}=\alpha^{7}+\alpha^{5}+\alpha^{3}+1$ & 63 \\
$\alpha^{10}+\alpha^{8}+\alpha^{7}+\alpha^{6}+\alpha^{4}+\alpha^{3}+1$ & $\alpha^{14}=\alpha^{7}+\alpha^{5}+\alpha^{3}+1$ & 66 \\
$\alpha^{7}+\alpha^{6}+\alpha^{4}+\alpha^{3}+\alpha^{2}+\alpha$ & $\alpha^{12}=\alpha^{7}+\alpha^{6}+\alpha^{5}+\alpha^{3}+\alpha+1$ & 67 \\
$\alpha^{5}+\alpha^{4}$ & $\alpha^{12}=\alpha^{7}+\alpha^{6}+\alpha^{5}+\alpha^{3}+\alpha+1$ & 74 \\
$\alpha^{7}$ & $\alpha^{10}=\alpha^{6}+\alpha^{5}+\alpha^{3}+\alpha^{2}+\alpha+1$ & 76 \\
$\alpha^{12}+\alpha^{8}+\alpha^{6}+\alpha^{4}+\alpha^{3}+\alpha^{2}+\alpha+1$ & $\alpha^{15}=\alpha^{5}+\alpha^{4}+\alpha^{2}+1$ & 77 \\
$\alpha^{4}+\alpha^{3}+1$ & $\alpha^{12}=\alpha^{7}+\alpha^{6}+\alpha^{5}+\alpha^{3}+\alpha+1$ & 98 \\
$\alpha^{10}+\alpha^{9}+\alpha^{8}+\alpha^{6}+\alpha^{5}+\alpha^{4}+\alpha+1$ & $\alpha^{13}=\alpha^{4}+\alpha^{3}+\alpha+1$ & 104 \\
$\alpha^{9}+\alpha^{7}+\alpha^{6}+\alpha^{3}+\alpha^{2}+1$ & $\alpha^{13}=\alpha^{4}+\alpha^{3}+\alpha+1$ & 116 \\
$\alpha^{10}+\alpha^{9}+\alpha^{8}+\alpha^{7}+\alpha^{5}+\alpha^{3}+\alpha+1$ & $\alpha^{15}=\alpha^{5}+\alpha^{4}+\alpha^{2}+1$ & 128 \\
$\alpha^{2}+\alpha$ & $\alpha^{13}=\alpha^{4}+\alpha^{3}+\alpha+1$ & 130 \\
$\alpha^{5}+\alpha^{4}+\alpha^{3}+\alpha$ & $\alpha^{15}=\alpha^{5}+\alpha^{4}+\alpha^{2}+1$ & 131 \\
$\alpha^{9}+\alpha^{7}+\alpha^{5}+\alpha^{3}+\alpha$ & $\alpha^{11}=\alpha^{2}+1$ & 133 \\
$\alpha^{6}+\alpha^{5}+\alpha^{3}+\alpha^{2}+\alpha$ & $\alpha^{12}=\alpha^{7}+\alpha^{6}+\alpha^{5}+\alpha^{3}+\alpha+1$ & 141 \\
$\alpha^{9}+\alpha^{8}+\alpha^{6}+\alpha^{4}+\alpha^{2}+\alpha$ & $\alpha^{15}=\alpha^{5}+\alpha^{4}+\alpha^{2}+1$ & 144 \\
$\alpha^{10}+\alpha^{8}+\alpha^{7}+\alpha^{6}+\alpha^{5}+\alpha^{4}+\alpha$ & $\alpha^{14}=\alpha^{7}+\alpha^{5}+\alpha^{3}+1$ & 152 \\
$\alpha^{9}+\alpha^{8}+\alpha^{7}+\alpha^{6}+\alpha^{5}+\alpha^{4}+\alpha^{2}$ & $\alpha^{13}=\alpha^{4}+\alpha^{3}+\alpha+1$ & 157 \\
$\alpha^{10}+\alpha^{8}+\alpha^{7}+\alpha^{6}+\alpha^{3}$ & $\alpha^{14}=\alpha^{7}+\alpha^{5}+\alpha^{3}+1$ & 162 \\
$\alpha^{12}+\alpha^{10}+\alpha^{8}+\alpha^{6}+\alpha^{2}+\alpha+1$ & $\alpha^{15}=\alpha^{5}+\alpha^{4}+\alpha^{2}+1$ & 164 \\
$\alpha^{10}+\alpha^{8}+\alpha^{7}+\alpha^{5}+\alpha^{4}+\alpha^{3}+\alpha$ & $\alpha^{15}=\alpha^{5}+\alpha^{4}+\alpha^{2}+1$ & 169 \\
$\alpha^{8}+\alpha^{7}+\alpha^{4}+\alpha^{3}+\alpha^{2}+1$ & $\alpha^{15}=\alpha^{5}+\alpha^{4}+\alpha^{2}+1$ & 174 \\
$\alpha^{7}+\alpha^{5}+\alpha^{4}$ & $\alpha^{15}=\alpha^{5}+\alpha^{4}+\alpha^{2}+1$ & 176 \\
$\alpha^{9}+\alpha^{7}+\alpha^{6}+\alpha^{5}+\alpha^{2}+\alpha$ & $\alpha^{14}=\alpha^{7}+\alpha^{5}+\alpha^{3}+1$ & 188 \\
$\alpha^{8}+\alpha^{7}+\alpha^{4}+\alpha^{2}+\alpha$ & $\alpha^{15}=\alpha^{5}+\alpha^{4}+\alpha^{2}+1$ & 258 \\
$\alpha^{7}+\alpha^{3}$ & $\alpha^{15}=\alpha^{5}+\alpha^{4}+\alpha^{2}+1$ & 277 \\
$\alpha^{11}+\alpha^{8}+\alpha^{6}+\alpha^{4}$ & $\alpha^{14}=\alpha^{7}+\alpha^{5}+\alpha^{3}+1$ & 280 \\
$\alpha^{12}+\alpha^{6}+\alpha+1$ & $\alpha^{15}=\alpha^{5}+\alpha^{4}+\alpha^{2}+1$ & 323 \\
$\alpha^{11}+\alpha^{9}+\alpha^{7}+\alpha^{6}+\alpha^{5}+\alpha^{4}+\alpha^{3}$ & $\alpha^{15}=\alpha^{5}+\alpha^{4}+\alpha^{2}+1$ & 427\\
\end{tabular}}

    \caption{Some representatives of elements of finite escape time}
     \label{table}
    \end{center}
\end{table}

\end{document}